\documentclass[reqno,12pt]{amsart}
\usepackage{amsfonts}
\usepackage{bbm}
\usepackage{}
\numberwithin{equation}{section}
\usepackage{color}

\usepackage{indentfirst}
\usepackage{color}
\usepackage{amssymb}
\usepackage{mathrsfs}
\usepackage{xy}
\xyoption{all}
\def\Ext{\mbox{\rm Ext}\,} \def\Hom{\mbox{\rm Hom}}  
 \def\fin{\hfill$\square$}   \def\mod{\mbox{\rm \textbf{mod}}\,}
\def\Ker{\mbox{\rm Ker}\,} \def\Coker{\mbox{\rm Coker}\,}
\def\cone{\mbox{\rm cone}}\def\cocone{\mbox{\rm cocone}}
\def\B{\mathcal {B}}\def\C{\mathcal {C}}
\def\A{\mathcal{A}} 
\def\Id{\mbox{\rm Id}\,} \def\Im{\mbox{\rm Im}\,} \def\add{\mbox{\rm add}\,}

\theoremstyle{plain}
\newtheorem{theorem}{\bf Theorem}[section]
\newtheorem{lemma}[theorem]{\bf Lemma}
\newtheorem{corollary}[theorem]{\bf Corollary}
\newtheorem{proposition}[theorem]{\bf Proposition}

\theoremstyle{definition}
\newtheorem{definition}[theorem]{\bf Definition}
\newtheorem{remark}[theorem]{\bf Remark}
\newtheorem{example}[theorem]{\bf Example}

\newtheorem{condition}[theorem]{\bf Condition}
\newcommand{\bt}{\begin{theorem}}
\newcommand{\et}{\end{theorem}}
\newcommand{\bl}{\begin{lemma}}
\newcommand{\el}{\end{lemma}}
\newcommand{\bd}{\begin{definition}}
\newcommand{\ed}{\end{definition}}
\newcommand{\bc}{\begin{corollary}}
\newcommand{\ec}{\end{corollary}}
\newcommand{\bp}{\begin{proof}}
\newcommand{\ep}{\end{proof}}
\newcommand{\bx}{\begin{example}}
\newcommand{\ex}{\end{example}}
\newcommand{\br}{\begin{remark}}
\newcommand{\er}{\end{remark}}
\newcommand{\be}{\begin{equation}}
\newcommand{\ee}{\end{equation}}
\newcommand{\ba}{\begin{align}}
\newcommand{\ea}{\end{align}}
\newcommand{\bn}{\begin{enumerate}}
\newcommand{\en}{\end{enumerate}}
\newcommand{\bcs}{\begin{cases}}
\newcommand{\ecs}{\end{cases}}

\makeatletter
\renewcommand{\section}{\@startsection{section}{1}{0mm}
  {-\baselineskip}{0.5\baselineskip}{\bf\leftline}}
\makeatother

\begin{document}

\title[Recollements of extriangulated categories]{Recollements of extriangulated categories}
\author[L. Wang, J. Wei, H. Zhang]{Li Wang, Jiaqun Wei,  Haicheng Zhang}
\address{Institute of Mathematics, School of Mathematical Sciences, Nanjing Normal University,
 Nanjing 210023, P. R. China.\endgraf}
\email{wl04221995@163.com (Wang); weijiaqun@njnu.edu.cn (Wei); zhanghc@njnu.edu.cn (Zhang).}


\begin{abstract} We give a simultaneous generalization of recollements of abelian categories and triangulated categories, which we call recollements of extriangulated categories. For a recollement  $(\mathcal{A}$, $\mathcal{B}$, $\mathcal{C})$ of extriangulated categories, we show that cotorsion pairs in $\mathcal{A}$ and $\mathcal{C}$ induce cotorsion pairs in $\mathcal{B}$ under certain conditions. As an application, our main result recovers a result given by Chen for recollements of triangulated categories, and it also shows a
new phenomena when it is applied to abelian categories.

\end{abstract}

\maketitle

\section{Introduction}
Abelian categories and triangulated categories are two fundamental structures in algebra and geometry. Recollements of triangulated categories were introduced by Be{\u\i}linson, Bernstein and Deligne \cite{BBD} in
connection with derived categories of sheaves on topological spaces with the idea that one triangulated category may be ``glued together" from two others. The recollements of abelian categories first appeared in the construction of the category of perverse sheaves on a singular space in \cite{BBD}. Recollements of abelian categories and triangulated categories play an
important role in algebraic geometry and representation theory, see for instance \cite{Fr},  \cite{Ang}, \cite{Ps}. In recollements of abelian categories, the relations with tilting modules and torsion pairs have been studied in \cite{Ma} and \cite{Ma2}, respectively.  Chen \cite{Chen} studied the relationship of cotorsion pairs among three
triangulated categories in a recollement of triangulated categories.


Recently, Nakaoka
and Palu \cite{Na} introduced an extriangulated category which is extracting properties on triangulated categories and exact categories. Recollements of abelian categories and triangulated categories are closely related, and they possess similar properties in many aspects. This inspires us to give a simultaneous generalization of recollements of abelian categories and triangulated categories, which we call recollements of extriangulated categories. Then we study the relationship of cotorsion pairs in a recollement of extriangulated categories. In order to achieve this goal, we need to consider the WIC Condition (cf. \cite[Condition 5.8]{Na}), introduce compatible morphisms, and then define left exact sequences, right exact sequences, left exact functors, and right exact functors in extriangulated categories.


The paper is organized as follows: we summarize some basic definitions and properties of extriangulated categories and exact functors in Section 2. In Section 3, we introduce the recollement of extriangulated categories and give some basic properties. Section 4 is devoted to giving conditions such that the glued pair with respect to cotorsion pairs in $\mathcal{A}$ and $\mathcal{C}$ is a cotorsion pair in $\mathcal{B}$ for a recollement $(\mathcal{A},\mathcal{B},\mathcal{C})$ of extriangulated categories. Moreover, we show that the conserve also holds for some special cotorsion pairs in $\mathcal{B}$. As an application, we recover the corresponding results in the triangulated category case.

\subsection{Conventions and notation.}
For an additive category $\mathscr{C}$, its subcategories are assumed to be {full} and closed under isomorphisms. A subcategory $\mathcal{D}$ of $\mathscr{C}$ is said to be {\em contravariantly finite} in $\mathscr{C}$ if for each object $M\in\mathscr{C}$, there exists a morphism $f:X\rightarrow M$ with $X\in \mathcal{D}$ such that $\mathscr{C}(\mathcal{D},f)$ is an epimorphism. Dually, one defines {\em covariantly
finite} subcategories in $\mathscr{C}$. Given an object $M\in\mathscr{C}$, we denote by $\add M$ the additive closure of $M$, that is, the full subcategory of $\mathscr{C}$ whose objects are the direct sums of direct summands of $M$.
Let $Q$ be a finite acyclic quiver, we denote by $S_i$ the one-dimensional simple (left) $kQ$-module associated to the vertex $i$ of $Q$, and denote by $P_i$ and $I_i$ the projective cover and injective envelop of $S_i$, respectively.

\section{Preliminaries}
\subsection{Extriangulated categories}
Let us recall some notions concerning extriangulated categories from \cite{Na}.

Let $\mathscr{C}$ be an additive category and let $\mathbb{E}$: $\mathscr{C}^{op}\times\mathscr{C}\rightarrow Ab$ be a biadditive functor. For any pair of objects $A$, $C\in\mathscr{C}$, an element $\delta\in \mathbb{E}(C,A)$ is called an {\em $\mathbb{E}$-extension}. The zero element $0\in\mathbb{E}(C,A)$ is called the {\em split $\mathbb{E}$-extension}.
 For any morphism $a\in \mathscr{C}(A,A')$ and $c\in {\mathscr{C}}(C',C)$, we have
$\mathbb{E}(C,a)(\delta)\in\mathbb{E}(C,A')$ and $\mathbb{E}(c,A)(\delta)\in\mathbb{E}(C',A).$
We simply denote them by $a_{\ast}\delta$ and $c^{\ast}\delta$,~respectively.  A morphism $(a,c)$: $\delta\rightarrow\delta'$ of $\mathbb{E}$-extensions is a pair of morphisms $a\in \mathscr{C}(A,A')$ and $c\in {\mathscr{C}}(C,C')$ satisfying the equality $a_{\ast}\delta=c^{\ast}\delta'$.

By Yoneda's lemma, any $\mathbb{E}$-extension $\delta\in \mathbb{E}(C,A)$ induces natural transformations
$$\delta_{\sharp}: \mathscr{C}(-,C)\rightarrow\mathbb{E}(-,A)~~\text{and}~~\delta^{\sharp}: \mathscr{C}(A,-)\rightarrow\mathbb{E}(C,-).$$ For any $X\in\mathscr{C}$, these
$(\delta_{\sharp})_X$ and $(\delta^{\sharp})_X$ are defined by
$(\delta_{\sharp})_X:  \mathscr{C}(X,C)\rightarrow\mathbb{E}(X,A), f\mapsto f^\ast\delta$ and $(\delta^{\sharp})_X:   \mathscr{C}(A,X)\rightarrow\mathbb{E}(C,X), g\mapsto g_\ast\delta$.

Two sequences of morphisms $A\stackrel{x}{\longrightarrow}B\stackrel{y}{\longrightarrow}C$ and $A\stackrel{x'}{\longrightarrow}B'\stackrel{y'}{\longrightarrow}C$ in $\mathscr{C}$ are said to be {\em equivalent} if there exists an isomorphism $b\in \mathscr{C}(B,B')$ such that the following diagram
$$\xymatrix{
  A \ar@{=}[d] \ar[r]^-{x} & B\ar[d]_{b}^-{\simeq} \ar[r]^-{y} & C\ar@{=}[d] \\
  A \ar[r]^-{x'} &B' \ar[r]^-{y'} &C  }$$ is commutative.
We denote the equivalence class of $A\stackrel{x}{\longrightarrow}B\stackrel{y}{\longrightarrow}C$ by $[A\stackrel{x}{\longrightarrow}B\stackrel{y}{\longrightarrow}C]$. In addition, for any $A,C\in\mathscr{C}$, we denote as
$$0=[A\stackrel{{1\choose0}}{\longrightarrow}A\oplus C\stackrel{(0~1)}{\longrightarrow}C].$$
For any two equivalence classes $[A\stackrel{x}{\longrightarrow}B\stackrel{y}{\longrightarrow}C]$ and $[A'\stackrel{x'}{\longrightarrow}B'\stackrel{y'}{\longrightarrow}C']$, we denote as
$$[A\stackrel{x}{\longrightarrow}B\stackrel{y}{\longrightarrow}C]\oplus[A'\stackrel{x'}{\longrightarrow}B'\stackrel{y'}{\longrightarrow}C']=
[A\oplus A'\stackrel{x\oplus x'}{\longrightarrow}B\oplus B'\stackrel{y\oplus y'}{\longrightarrow}C\oplus C'].$$

\begin{definition}
Let $\mathfrak{s}$ be a correspondence which associates an equivalence class $\mathfrak{s}(\delta)=[A\stackrel{x}{\longrightarrow}B\stackrel{y}{\longrightarrow}C]$ to any $\mathbb{E}$-extension $\delta\in\mathbb{E}(C,A)$ . This $\mathfrak{s}$ is called a {\em realization} of $\mathbb{E}$ if for any morphism $(a,c):\delta\rightarrow\delta'$ with $\mathfrak{s}(\delta)=[\Delta_{1}]$ and $\mathfrak{s}(\delta')=[\Delta_{2}]$, there is a commutative diagram as follows:
$$\xymatrix{
\Delta_{1}\ar[d] & A \ar[d]_-{a} \ar[r]^-{x} & B  \ar[r]^{y}\ar[d]_-{b} & C \ar[d]_-{c}    \\
 \Delta_{2}&A\ar[r]^-{x'} & B \ar[r]^-{y'} & C .   }
$$  A realization $\mathfrak{s}$ of $\mathbb{E}$ is said to be {\em additive} if it satisfies the following conditions:

(a) For any $A,~C\in\mathscr{C}$, the split $\mathbb{E}$-extension $0\in\mathbb{E}(C,A)$ satisfies $\mathfrak{s}(0)=0$.

(b) $\mathfrak{s}(\delta\oplus\delta')=\mathfrak{s}(\delta)\oplus\mathfrak{s}(\delta')$ for any pair of $\mathbb{E}$-extensions $\delta$ and $\delta'$.
\end{definition}

Let $\mathfrak{s}$ be an additive realization of $\mathbb{E}$. If $\mathfrak{s}(\delta)=[A\stackrel{x}{\longrightarrow}B\stackrel{y}{\longrightarrow}C]$, then the sequence $A\stackrel{x}{\longrightarrow}B\stackrel{y}{\longrightarrow}C$ is called a {\em conflation}, $x$ is called an {\em inflation} and $y$ is called a {\em deflation}.
In this case, we say $A\stackrel{x}{\longrightarrow}B\stackrel{y}{\longrightarrow}C\stackrel{\delta}\dashrightarrow$ is an $\mathbb{E}$-triangle.
We will write $A=\cocone(y)$ and $C=\cone(x)$ if necessary. We say an $\mathbb{E}$-triangle is {\em splitting} if it
realizes 0.

\begin{definition}
(\cite[Definition 2.12]{Na})\label{F}
We call the triplet $(\mathscr{C}, \mathbb{E},\mathfrak{s})$ an {\em extriangulated category} if it satisfies the following conditions:\\
$\rm(ET1)$ $\mathbb{E}$: $\mathscr{C}^{op}\times\mathscr{C}\rightarrow Ab$ is a biadditive functor.\\
$\rm(ET2)$ $\mathfrak{s}$ is an additive realization of $\mathbb{E}$.\\
$\rm(ET3)$ Let $\delta\in\mathbb{E}(C,A)$ and $\delta'\in\mathbb{E}(C',A')$ be any pair of $\mathbb{E}$-extensions, realized as
$\mathfrak{s}(\delta)=[A\stackrel{x}{\longrightarrow}B\stackrel{y}{\longrightarrow}C]$, $\mathfrak{s}(\delta')=[A'\stackrel{x'}{\longrightarrow}B'\stackrel{y'}{\longrightarrow}C']$. For any commutative square in $\mathscr{C}$
$$\xymatrix{
  A \ar[d]_{a} \ar[r]^{x} & B \ar[d]_{b} \ar[r]^{y} & C \\
  A'\ar[r]^{x'} &B'\ar[r]^{y'} & C'}$$
there exists a morphism $(a,c)$: $\delta\rightarrow\delta'$ which is realized by $(a,b,c)$.\\
$\rm(ET3)^{op}$~Dual of $\rm(ET3)$.\\
$\rm(ET4)$~Let $\delta\in\mathbb{E}(D,A)$ and $\delta'\in\mathbb{E}(F,B)$ be $\mathbb{E}$-extensions realized by
$A\stackrel{f}{\longrightarrow}B\stackrel{f'}{\longrightarrow}D$ and $B\stackrel{g}{\longrightarrow}C\stackrel{g'}{\longrightarrow}F$, respectively.
Then there exist an object $E\in\mathscr{C}$, a commutative diagram
\begin{equation}\label{2.1}
\xymatrix{
  A \ar@{=}[d]\ar[r]^-{f} &B\ar[d]_-{g} \ar[r]^-{f'} & D\ar[d]^-{d} \\
  A \ar[r]^-{h} & C\ar[d]_-{g'} \ar[r]^-{h'} & E\ar[d]^-{e} \\
   & F\ar@{=}[r] & F   }
\end{equation}
in $\mathscr{C}$, and an $\mathbb{E}$-extension $\delta''\in \mathbb{E}(E,A)$ realized by $A\stackrel{h}{\longrightarrow}C\stackrel{h'}{\longrightarrow}E$, which satisfy the following compatibilities:\\
$(\textrm{i})$ $D\stackrel{d}{\longrightarrow}E\stackrel{e}{\longrightarrow}F$ realizes $\mathbb{E}(F,f')(\delta')$,\\
$(\textrm{ii})$ $\mathbb{E}(d,A)(\delta'')=\delta$,\\
$(\textrm{iii})$ $\mathbb{E}(E,f)(\delta'')=\mathbb{E}(e,B)(\delta')$.\\
$\rm(ET4)^{op}$ Dual of $\rm(ET4)$.
\end{definition}

Let $\mathscr{C}$ be an extriangulated category, and $\mathcal{D},\mathcal{D}'\subseteq\mathscr{C}$. We write $\mathcal{D}\ast\mathcal{D}'$ for the full subcategory of objects $X$ admitting an $\mathbb{E}$-triangle $D\stackrel{}{\longrightarrow}X\stackrel{}{\longrightarrow}D'\stackrel{}\dashrightarrow$ with $D\in\mathcal{D}$ and $D'\in\mathcal{D}'$. A subcategory $\mathcal{D}$ of  $\mathscr{C}$ is {\em extension-closed}, if $\mathcal{D}\ast\mathcal{D}=\mathcal{D}$.
An object $P$ in $\mathscr{C}$ is called {\em projective} if for any conflation $A\stackrel{x}{\longrightarrow}B\stackrel{y}{\longrightarrow}C$ and any morphism $c$ in $\mathscr{C}(P,C)$, there exists $b$ in $\mathscr{C}(P,B)$ such that $yb=c$. We denote the full subcategory of projective objects in $\mathscr{C}$ by $\mathcal{P}(\mathscr{C})$. Dually, the {\em injective} objects are defined, and the full subcategory of injective objects in $\mathscr{C}$ is denoted by $\mathcal{I}(\mathscr{C})$. We say that $\mathscr{C}$ {\em has enough projectives} if for any object $M\in\mathscr{C}$, there exists an $\mathbb{E}$-triangle $A\stackrel{}{\longrightarrow}P\stackrel{}{\longrightarrow}M\stackrel{}\dashrightarrow$ satisfying $P\in\mathcal{P}(\mathscr{C})$. Dually, we define that $\mathscr{C}$ {\em has enough injectives}. In particular, if $\mathscr{C}$ is a triangulated category, then $\mathscr{C}$  has enough projectives and injectives with $\mathcal{P}(\mathscr{C})$ and $\mathcal{I}(\mathscr{C})$ consisting of zero objects.

\begin{example}
(a)  Exact categories, triangulated categories and extension-closed subcategories of triangulated categories are
extriangulated categories. (cf. \cite{Na})

(b) Let $\mathscr{C}$ be an extriangulated category. Then $\mathscr{C}/(\mathcal{P}(\mathscr{C})\cap \mathcal{I}(\mathscr{C}))$ is an extriangulated category which is neither exact nor triangulated in general (cf. \cite[Proposition 3.30]{Na}).
\end{example}

\begin{proposition}\label{exact} \cite[Proposition 3.3]{Na}
Let $\mathscr{C}$ be an extriangulated category. For any $\mathbb{E}$-triangle $A\stackrel{}{\longrightarrow}B\stackrel{}{\longrightarrow}C\stackrel{\delta}\dashrightarrow$, the following sequences of natural transformations are exact.
$$\mathscr{C}(C,-)\rightarrow \mathscr{C}(B,-)\rightarrow \mathscr{C}(A,-)\stackrel{\delta^{\sharp}}\rightarrow\mathbb{E}(C,-)\rightarrow\mathbb{E}(B,-),$$
$$\mathscr{C}(-,A)\rightarrow \mathscr{C}(-,B)\rightarrow \mathscr{C}(-,C)\stackrel{\delta_{\sharp}}\rightarrow\mathbb{E}(-,A)\rightarrow\mathbb{E}(-,B).$$
\end{proposition}

Let $\mathscr{C}$ be an extriangulated category. A finite sequence $$X_{n}\stackrel{d_{n}}{\longrightarrow}X_{n-1}\stackrel{d_{n-1}}{\longrightarrow}\cdots \stackrel{d_{2}}{\longrightarrow} X_{1}\stackrel{d_{1}}{\longrightarrow}X_{0}$$ in $\mathscr{C}$ is said to be an {\em $\mathbb{E}$-triangle sequence}, if there exist $\mathbb{E}$-triangles $X_{n}\stackrel{d_{n}}{\longrightarrow}X_{n-1}\stackrel{f_{n-1}}{\longrightarrow}K_{n-1}\stackrel{}\dashrightarrow$,
$K_{i+1}\stackrel{g_{i+1}}{\longrightarrow}X_{i}\stackrel{f_{i}}{\longrightarrow}K_{i}\stackrel{}\dashrightarrow,~~1<i<n-1$,
and $K_{2}\stackrel{g_{2}}{\longrightarrow}X_{1}\stackrel{d_1}{\longrightarrow}X_0\stackrel{}\dashrightarrow$ such that $d_i=g_if_i$ for any $1<i<n$.

Two morphism sequences $\eta_{1}:A\stackrel{f}{\longrightarrow}B\stackrel{g}{\longrightarrow}C$ and $\eta_{2}:A'\stackrel{f'}{\longrightarrow}B'\stackrel{g'}{\longrightarrow}C'$ in $\mathscr{C}$ are said to be {\em isomorphic}, denoted by $\eta_{1}\simeq \eta_{2}$, if there are isomorphisms $x:A\rightarrow A'$, $y:B\rightarrow B'$ and $z:C\rightarrow C'$ in $\mathscr{C}$ such that  $yf=f'x$ and $zg=g'y$.
\begin{lemma}\label{2.5} 

$(1)$ Let $\eta_{1}$ and $\eta_{2}$ be morphism sequences in $\mathscr{C}$ such that $\eta_{1}\simeq \eta_{2}$. Then $\eta_{1}$ is a conflation if and only if $\eta_{2}$ is a conflation.

$(2)$ Let $A\stackrel{f}{\longrightarrow}B\stackrel{g}{\longrightarrow}C\stackrel{}\dashrightarrow$ be an $\mathbb{E}$-triangle in $\mathscr{C}$. Then $f$ is an isomorphism if and only if $C\cong0$. Similarly, $g$ is an isomorphism if and only if $A\cong0$.
\end{lemma}
\begin{proof} It is easily proved by \cite[Corollary 3.6]{Na} and \cite[Proposition 3.7]{Na}.
\end{proof}

\subsection{Exact functors}
In what follows, we will always assume that $\mathscr{C}$ is an extriangulated category. In addition, we assume the following conditions for the rest of this paper (see \cite[Condition 5.8]{Na}).

\begin{condition}\label{WIC} (WIC)
(1) Let $f:X\rightarrow Y$ and  $g:Y\rightarrow Z$ be any composable pair of morphisms in $\mathscr{C}$. If $gf$ is an inflation, then $f$ is an inflation.

(2) Let $f:X\rightarrow Y$ and  $g:Y\rightarrow Z$ be any composable pair of morphisms in $\mathscr{C}$. If $gf$ is a deflation, then $g$ is a deflation.
\end{condition}
\begin{remark} If $\mathscr{C}$ is a triangulated category or weakly idempotent complete exact category (\cite[Proposition 7.6]{Bu}), then Condition \ref{WIC} is satisfied.
\end{remark}

\begin{definition}\label{right} A morphism $f$ in $\mathscr{C}$ is called {\em compatible}, if ``$f$ is both an inflation and a deflation" implies $f$ is an isomorphism. That is, the class of compatible morphisms is the following class consisting of the morphisms in $\mathscr{C}$
$$\{f~|~f~\text{is~not~an~inflation},~\text{or}~f~\text{is~not~a~deflation},~\text{or}~f~\text{is~an~isomorphism}\}.$$
\end{definition}

It is clear that all morphisms are compatible in an exact category. While, the compatible morphisms in a triangulated category $\mathscr{C}$ are just the isomorphisms in $\mathscr{C}$.

\begin{definition}\label{right} A sequence $A\stackrel{f}{\longrightarrow}B\stackrel{g}{\longrightarrow}C$ in $\mathscr{C}$ is said to be {\em right exact} if
there exists an $\mathbb{E}$-triangle $K\stackrel{h_{2}}{\longrightarrow}B\stackrel{g}{\longrightarrow}C\stackrel{}\dashrightarrow$ and a deflation $h_{1}:A\rightarrow K$ which is compatible, such that $f=h_2h_1$. Dually one can also define the {\em left exact} sequences.

A $4$-term $\mathbb{E}$-triangle sequence $A{\stackrel{f}\longrightarrow}B\stackrel{g}{\longrightarrow}C\stackrel{h}{\longrightarrow}D$ is called {\em right exact} (resp. {\em left exact}) if there exist $\mathbb{E}$-triangles $A\stackrel{f}{\longrightarrow}B\stackrel{g_1}{\longrightarrow}K\stackrel{}\dashrightarrow$
and $K\stackrel{g_{2}}{\longrightarrow}C\stackrel{h}{\longrightarrow}D\stackrel{}\dashrightarrow$ such that $g=g_2g_1$ and $g_1$ (resp. $g_2$) is compatible.
\end{definition}

For the convenience of statements, given a morphism $f:A\rightarrow B$ in $\mathscr{C}$, we denote by $\Phi_{f}$ the set consisting of all pairs $(h_{1},h_{2})$ such that $h_{1}:A\rightarrow K$ is a deflation, $h_{2}:K\rightarrow B$ is an inflation and $f=h_{2}h_{1}$.
\begin{lemma}\label{zero} Let $\eta:A\stackrel{f}{\longrightarrow}B\stackrel{g}{\longrightarrow}C$ be a right exact sequence in $\mathscr{C}$.

$(1)$  If $f$ is an inflation, then $\eta$ is a conflation.

$(2)$ If $A=0$, then $g$ is an isomorphism.

\end{lemma}
\begin{proof}
$(1)$ Since $\eta$ is right exact, there is an $\mathbb{E}$-triangle $K\stackrel{h_{2}}{\longrightarrow}B\stackrel{g}{\longrightarrow}C\stackrel{}\dashrightarrow$ and a compatible morphism $h_{1}$ such that $(h_{1},h_{2})\in\Phi_{f}$. Since $f$ is an inflation, by Condition \ref{WIC}(1), we obtain that $h_1$ is an inflation. Since $h_1$ is also a deflation, we get that $h_1$ is an isomorphism. Then by Lemma \ref{2.5}(1), we have that $\eta$ is a conflation.

$(2)$ Noting that $0\rightarrow B$ is an inflation, we obtain that $0\rightarrow B\stackrel{g}\rightarrow C$ is an $\mathbb{E}$-triangle. Then by Lemma \ref{2.5}(2), we get that $g$ is an isomorphism.
\end{proof}

We omit the dual statement of Lemma \ref{zero}.

\begin{remark}\label{2.8} (1) A sequence $\eta: A\stackrel{f}{\longrightarrow}B\stackrel{g}{\longrightarrow}C$ is both left exact and right exact if and only  if $\eta$ is a conflation.

(2) If $\mathscr{C}$ is an abelian category, then $A\stackrel{f}{\longrightarrow}B\stackrel{g}{\longrightarrow}C$ is right exact if and only if $A\stackrel{f}{\longrightarrow}B\stackrel{g}{\longrightarrow}C{\longrightarrow}0$ is exact. Similarly, $A\stackrel{f}{\longrightarrow}B\stackrel{g}{\longrightarrow}C$ is left exact if and only if $0\stackrel{} {\longrightarrow} A\stackrel{f}{\longrightarrow}B\stackrel{g}{\longrightarrow}C$ is exact. If $\mathscr{C}$ is a triangulated category with the suspension  functor [1]. Then $\eta: A\stackrel{f}{\longrightarrow}B\stackrel{g}{\longrightarrow}C$  is right exact if and only if $A\stackrel{f}{\longrightarrow}B\stackrel{g}{\longrightarrow}C{\longrightarrow}A[1]$ is a triangle if and only if $\eta$ is left exact.


\end{remark}

Let us give the notion of right (left) exact functors in extriangulated categories.

\begin{definition}\label{right exact} Let $(\mathcal{A},\mathbb{E}_{\mathcal{A}},\mathfrak{s}_{\mathcal{A}})$ and $(\mathcal{B},\mathbb{E}_{\mathcal{B}},\mathfrak{s}_{\mathcal{B}})$ be extriangulated categories. An additive covariant functor $F:\mathcal{A}\rightarrow \mathcal{B}$ is called a {\em right exact functor} if it satisfies the following conditions
\begin{itemize}
   \item [(1)] If $f$ is a compatible morphism in $\A$, then $Ff$ is compatible in $\B$.
   \item [(2)] If $A\stackrel{a}{\longrightarrow}B\stackrel{b}{\longrightarrow}C$ is right exact in $\mathcal{A}$, then $FA\stackrel{Fa}{\longrightarrow}FB\stackrel{Fb}{\longrightarrow}FC$ is right exact in $\mathcal{B}$. (Then for any $\mathbb{E}_{\mathcal{A}}$-triangle $A\stackrel{f}{\longrightarrow}B\stackrel{g}{\longrightarrow}C\stackrel{\delta}\dashrightarrow$, there exists an $\mathbb{E}_{\mathcal{B}}$-triangle $A'\stackrel{x}{\longrightarrow}FB\stackrel{Fg}{\longrightarrow}FC\stackrel{}\dashrightarrow$ such that $Ff=xy$ and $y: FA\rightarrow A'$ is a deflation and compatible. Moreover, $A'$ is uniquely determined up to isomorphism.)

  \item [(3)] There exists a natural transformation $$\eta=\{\eta_{(C,A)}:\mathbb{E}_{\mathcal{A}}(C,A)\longrightarrow\mathbb{E}_{\mathcal{B}}(F^{op}C,A')\}_{(C,A)\in{\A}^{\rm op}\times\A}$$ such that $\mathfrak{s}_{\mathcal{B}}(\eta_{(C,A)}(\delta))=[A'\stackrel{x}{\longrightarrow}FB\stackrel{Fg}{\longrightarrow}FC]$.

 \end{itemize}
Dually, we define the {\em left exact functor} between two extriangulated categories.
\end{definition}

The {\em extriangulated functor} between two extriangulated categories has been defined in \cite{Ben}.  For our requirements, we modify the definition as the following
\begin{definition}\label{exact functor}
Let $(\mathcal{A},\mathbb{E}_{\mathcal{A}},\mathfrak{s}_{\mathcal{A}})$ and $(\mathcal{B},\mathbb{E}_{\mathcal{B}},\mathfrak{s}_{\mathcal{B}})$ be extriangulated categories.  We say an additive covariant functor $F:\mathcal{A}\rightarrow \mathcal{B}$ is an {\em exact functor} if the following conditions hold.
\begin{itemize}
  \item [(1)] If $f$ is a compatible morphism in $\A$, then $Ff$ is compatible in $\B$.
  \item [(2)] There exists a natural transformation $$\eta=\{\eta_{(C,A)}\}_{(C,A)\in{\A}^{\rm op}\times\A}:\mathbb{E}_{\mathcal{A}}(-,-)\Rightarrow\mathbb{E}_{\mathcal{B}}(F^{\rm op}-,F-).$$
  \item [(3)] If $\mathfrak{s}_{\mathcal{A}}(\delta)=[A\stackrel{x}{\longrightarrow}B\stackrel{y}{\longrightarrow}C]$, then $\mathfrak{s}_{\mathcal{B}}(\eta_{(C,A)}(\delta))=[F(A)\stackrel{F(x)}{\longrightarrow}F(B)\stackrel{F(y)}{\longrightarrow}F(C)]$.

  \end{itemize}

\end{definition}


\begin{proposition}\label{2.10} Let $(\mathcal{A},\mathbb{E}_{\mathcal{A}},\mathfrak{s}_{\mathcal{A}})$ and $(\mathcal{B},\mathbb{E}_{\mathcal{B}},\mathfrak{s}_{\mathcal{B}})$ be extriangulated categories. An additive covariant functor $F: \mathcal{A}\rightarrow\mathcal{B}$ is exact if and only if $F$ is both left exact and right exact.
\end{proposition}
\begin{proof}
We only need to prove the sufficiency.  Let  $A\stackrel{f}{\longrightarrow}B\stackrel{g}{\longrightarrow}C\stackrel{\delta}\dashrightarrow$ be an $\mathbb{E}_{\A}$-triangle. Since $F$ is both left exact and right exact, we obtain that $FA\stackrel{Ff}{\longrightarrow}FB\stackrel{Fg}{\longrightarrow}FC$ is both left exact and right exact. By Remark \ref{2.8}(1),
$FA\stackrel{Ff}{\longrightarrow}FB\stackrel{Fg}{\longrightarrow}FC$ is a conflation.
\end{proof}
\begin{remark}If the categories $\mathcal{A}$ and $\mathcal{B}$ are abelian, Definition \ref{right exact} coincides with the usual right exact functor in abelian categories, and Definition \ref{exact functor} coincides with the usual exact functor. If the categories $\mathcal{A}$ and $\mathcal{B}$ are triangulated, by Remark \ref{2.8} and Proposition \ref{2.10}, we know that $F$ is a left exact functor if and only if $F$ is a triangle functor if and only if  $F$ is a right exact functor.
\end{remark}


\begin{lemma}\label{Iso} Let $(\mathcal{A},\mathbb{E}_{\mathcal{A}},\mathfrak{s}_{\mathcal{A}})$ and $(\mathcal{B},\mathbb{E}_{\mathcal{B}},\mathfrak{s}_{\mathcal{B}})$ be extriangulated categories and $F:\mathcal{A}\rightarrow \mathcal{B}$ be a functor which admits a right adjoint functor $G$.

$(1)$ If $\mathcal{A}$ has enough projectives and $F$ is an exact functor which preserves projectives, then $\mathbb{E}_{\mathcal{B}}(FX,Y)\cong\mathbb{E}_{\mathcal{A}}(X,GY)$ for any $X\in\mathcal{A}$ and $Y\in\mathcal{B}$.

$(2)$ If $\mathcal{B}$ has enough injectives and $G$ is an exact functor which preserves injectives, then $\mathbb{E}_{\mathcal{B}}(FX,Y)\cong\mathbb{E}_{\mathcal{A}}(X,GY)$ for any $X\in\mathcal{A}$ and $Y\in\mathcal{B}$.
\end{lemma}
\begin{proof} (1) For any $X\in \mathcal{A}$, there exists an $\mathbb{E}_{\A}$-triangle \begin{equation}\label{1}\xymatrix{M\ar[r]^-{}&P\ar[r]^-{}&X\ar@{-->}[r]^{}&}\end{equation} with $P\in \mathcal{P({\mathcal{A}})}$. Since $F$ is an exact functor, and $F$ preserves projectives, we obtain the following $\mathbb{E}_{\B}$-triangle
\begin{equation}\label{2}\xymatrix{FM\ar[r]^-{}&FP\ar[r]^-{}&FX\ar@{-->}[r]&}\end{equation}
with $FP\in \mathcal{P}(\mathcal{B})$. Applying the functors $\Hom_{\mathcal{A}}(-,GY)$ and $\Hom_{\mathcal{B}}(-,Y)$ to (\ref{1}) and (\ref{2}), respectively, we get the following commutative diagram
$$\xymatrix{
  \Hom_{\mathcal{A}}(P,GY)\ar[d] \ar[r] & \Hom_{\mathcal{A}}(M,GY)  \ar[d]\ar[r] &  \mathbb{E}_{\mathcal{A}}(X,GY)\ar[d]  \ar[r] & 0\\
 \Hom_{\mathcal{B}}(FP,Y) \ar[r] & \Hom_{\mathcal{B}}(FM,Y)\ar[r] &  \mathbb{E}_{\mathcal{B}}(FX,Y) \ar[r]&  0
 .}$$
By Five-Lemma, we obtain that $\mathbb{E}_{\mathcal{B}}(FX,Y)\cong\mathbb{E}_{\mathcal{A}}(X,GY)$.

(2) It is similar to (1).
\end{proof}

The following lemma is well-known. For the convenience of the reader we give a short proof.
\begin{lemma}\label{adjoint} Let $\mathcal{A}$ and $\mathcal{B}$ be two categories and $F:\mathcal{A}\rightarrow \mathcal{B}$ be a functor which admits a right adjoint functor $G$. Let $\eta:\Id_{\mathcal{A}}\Rightarrow GF$ be the unit and $\epsilon:FG\Rightarrow\Id_{\mathcal{B}}$ be the counit.

$(1)$ $\Id_{FX}=\epsilon_{FX}F{(\eta_{X})}$ for any $X\in\mathcal{A}$.

$(2)$ $\Id_{GY}=G(\epsilon_{Y}){\eta_{GY}}$ for any $Y\in\mathcal{B}$.
\end{lemma}
\begin{proof} (1) Let $\eta:\Hom_{\mathcal{B}}(F-,-)\Rightarrow\Hom_{\mathcal{A}}(-,G-)$ be the adjoint isomorphism. For any $X\in\mathcal{A}$, consider the following commutative diagram
$$\xymatrix{
  \Hom_{\B}(FGFX,FX)\ar[d]_{(F\eta_{X})^*} \ar[r]^{\eta_{GFX,FX}} & \Hom_{\A}(GFX,GFX) \ar[d]^{\eta_X^\ast} \\
  \Hom_{\B}(FX,FX) \ar[r]^{\eta_{X,FX}} & \Hom_{\A}(X,GFX)   .}
$$
It follows that $\eta_{X,FX}(\epsilon_{FX}F{(\eta_{X})})=(\eta_{GFX,FX}(\epsilon_{FX}))\eta_{X}=\eta_{X}=\eta_{X,FX}(\Id_{FX})$.
Hence, we have that $\Id_{FX}=\epsilon_{FX}F({\eta_{X}})$. The proof of (2) is similar. \end{proof}

\section{Recollements}
Let us introduce the concepts of recollements of extriangulated categories.

\begin{definition}\label{recollement}
Let $\mathcal{A}$, $\mathcal{B}$ and $\mathcal{C}$ be three extriangulated categories. A \emph{recollement} of $\mathcal{B}$ relative to
$\mathcal{A}$ and $\mathcal{C}$, denoted by ($\mathcal{A}$, $\mathcal{B}$, $\mathcal{C}$), is a diagram
\begin{equation}\label{recolle}
  \xymatrix{\mathcal{A}\ar[rr]|{i_{*}}&&\ar@/_1pc/[ll]|{i^{*}}\ar@/^1pc/[ll]|{i^{!}}\mathcal{B}
\ar[rr]|{j^{\ast}}&&\ar@/_1pc/[ll]|{j_{!}}\ar@/^1pc/[ll]|{j_{\ast}}\mathcal{C}}
\end{equation}
given by two exact functors $i_{*},j^{\ast}$, two right exact functors $i^{\ast}$, $j_!$ and two left exact functors $i^{!}$, $j_\ast$, which satisfies the following conditions:
\begin{itemize}
  \item [(R1)] $(i^{*}, i_{\ast}, i^{!})$ and $(j_!, j^\ast, j_\ast)$ are adjoint triples.
  \item [(R2)] $\Im i_{\ast}=\Ker j^{\ast}$.
  \item [(R3)] $i_\ast$, $j_!$ and $j_\ast$ are fully faithful.
  \item [(R4)] For each $X\in\mathcal{B}$, there exists a left exact $\mathbb{E}_{\B}$-triangle sequence
  \begin{equation}\label{first}
  \xymatrix{i_\ast i^! X\ar[r]^-{\theta_X}&X\ar[r]^-{\vartheta_X}&j_\ast j^\ast X\ar[r]&i_\ast A}
   \end{equation}
  with $A\in \mathcal{A}$, where $\theta_X$ and  $\vartheta_X$ are given by the adjunction morphisms.
  \item [(R5)] For each $X\in\mathcal{B}$, there exists a right exact $\mathbb{E}_{\B}$-triangle sequence
  \begin{equation}\label{second}
  \xymatrix{i_\ast\ar[r] A' &j_! j^\ast X\ar[r]^-{\upsilon_X}&X\ar[r]^-{\nu_X}&i_\ast i^\ast X&}
   \end{equation}
 with $A'\in \mathcal{A}$, where $\upsilon_X$ and $\nu_X$ are given by the adjunction morphisms.
\end{itemize}
\end{definition}

\begin{remark} (1) If the categories $\mathcal{A}$, $\mathcal{B}$ and $\mathcal{C}$ are abelian, then Definition \ref{recollement} coincides with the definition of recollement of abelian categories (cf. \cite{Fr}, \cite{Ps}, \cite{Ma2}).

(2) If the categories $\mathcal{A}$, $\mathcal{B}$ and $\mathcal{C}$ are triangulated, then Definition \ref{recollement} coincides with the definition of recollement of triangulated categories (cf. \cite{BBD}).
\end{remark}

Now, we collect some properties of recollement of extriangulated categories, which will be used in the sequel.

\begin{lemma}\label{CY} Let ($\mathcal{A}$, $\mathcal{B}$, $\mathcal{C}$) be a recollement of extriangulated categories as (\ref{recolle}).

$(1)$ All the natural transformations
$$i^{\ast}i_{\ast}\Rightarrow\Id_{\A},~\Id_{\A}\Rightarrow i^{!}i_{\ast},~\Id_{\C}\Rightarrow j^{\ast}j_{!},~j^{\ast}j_{\ast}\Rightarrow\Id_{\C}$$
are natural isomorphisms.

$(2)$ $i^{\ast}j_!=0$ and $i^{!}j_\ast=0$.

$(3)$ $i^{\ast}$ preserves projective objects and $i^{!}$ preserves injective objects.

$(3')$ $j_{!}$ preserves projective objects and $j_{\ast}$ preserves injective objects.

$(4)$ If $i^{!}$ (resp. $j_{\ast}$) is  exact, then $i_{\ast}$ (resp. $j^{\ast}$) preserves projective objects.

$(4')$ If $i^{\ast}$ (resp. $j_{!}$) is  exact, then $i_{\ast}$ (resp. $j^{\ast}$) preserves injective objects.

$(5)$ If $\mathcal{B}$ has enough projectives, then $\mathcal{A}$ has enough projectives and $\mathcal{P}(\mathcal{A})=\add(i^{\ast}(\mathcal{P}(\mathcal{B})))$; if $\mathcal{B}$ has enough injectives, then $\mathcal{A}$ has enough injectives and $\mathcal{I}(\mathcal{A})=\add(i^{!}(\mathcal{I}(\mathcal{B})))$.

$(6)$  If $\mathcal{B}$ has enough projectives and $j_{\ast}$ is exact, then $\mathcal{C}$ has enough projectives and $\mathcal{P}(\mathcal{C})=\add(j^{\ast}(\mathcal{P}(\mathcal{B})))$; if $\mathcal{B}$ has enough injectives and $j_{!}$ is exact, then $\mathcal{C}$ has enough injectives and $\mathcal{I}(\mathcal{C})=\add(j^{\ast}(\mathcal{I}(\mathcal{B})))$.


$(7)$ If $\mathcal{B}$ has enough projectives and $i^{!}$ is  exact, then $\mathbb{E}_{\mathcal{B}}(i_{\ast}X,Y)\cong\mathbb{E}_{\mathcal{A}}(X,i^{!}Y)$ for any $X\in\mathcal{A}$ and $Y\in\mathcal{B}$.

$(7')$  If $\mathcal{C}$ has enough projectives and $j_{!}$ is  exact, then $\mathbb{E}_{\mathcal{B}}(j_{!}Z,Y)\cong\mathbb{E}_{\mathcal{C}}(Z,j^{\ast}Y)$ for any $Y\in\mathcal{B}$ and $Z\in\mathcal{C}$.

$(8)$  If $i^{\ast}$ is exact, then $j_{!}$ is  exact.

$(8')$ If $i^{!}$ is exact, then $j_{\ast}$ is exact.
\end{lemma}
\begin{proof} $(1)$ The proof follows from the fact that $i_\ast$, $j_!$ and $j_\ast$ are fully faithful.

$(2)$ For any $X\in \mathcal{C}$, since $j^*i_*=0$, we have that
\begin{align*}
\Hom_{\mathcal{A}}(i^{\ast}j_!X,i^{\ast}j_!X)&\cong\Hom_{\mathcal{B}}(j_!X,i_{\ast}i^{\ast}j_!X)\\
 &\cong\Hom_{\mathcal{C}}(X,j^{\ast}i_{\ast}i^{\ast}j_!X)\\
 &=0,
\end{align*}which follows that $i^{\ast}j_!=0$. Similarly,  $i^{!}j_\ast=0$.

(3) Let $P\in \mathcal{P}(\mathcal{B})$, we need to show that $i^{\ast}P\in \mathcal{P}(\mathcal{A})$. Let
\begin{equation}\label{3}\xymatrix{X\ar[r]^-{f}&Y\ar[r]^-{g}&Z\ar@{-->}[r]&}\end{equation} be an arbitrary $\mathbb{E}_{\mathcal{A}}$-triangle
and $h\in\Hom_{\mathcal{A}}(i^*P,Z)$. Applying $i_{\ast}$ to (\ref{3}), we have an $\mathbb{E}_{\mathcal{B}}$-triangle
\begin{equation}\label{4}\xymatrix{i_{\ast}X\ar[r]^-{i_{\ast}f}&i_{\ast}Y\ar[r]^-{i_{\ast}g}&i_{\ast}Z\ar@{-->}[r]&.}\end{equation}
Applying the functors $\Hom_{\mathcal{A}}(i^{\ast}P,-)$ and $\Hom_{\mathcal{B}}(P,-)$ to (\ref{3}) and (\ref{4}), respectively, we obtain the following commutative diagram
$$\xymatrix{
  \Hom(i^{\ast}P,Y)\ar[d]_{\eta_{P,Y}} \ar[r] & \Hom(i^{\ast}P,Z)  \ar[d]_{\eta_{P,Z}}\ar[r] &  \mathbb{E}_{\mathcal{A}}(i^{\ast}P,X)\ar@{-->}[d]_{l} & \\
 \Hom(P,i_{\ast}Y) \ar[r] & \Hom(P,i_{\ast}Z)\ar[r] &  \mathbb{E}_{\mathcal{B}}(P,i_{\ast}X). &
 }$$
Since $P$ is projective in $\mathcal{B}$ and then $\mathbb{E}_{\mathcal{B}}(P,i_{\ast}X)=0$, there exists a morphism $t:P\rightarrow i_{\ast}Y$ such that $i_{\ast}g(t)=\eta_{P,Z}(h)$. It follows that
$$g(\eta_{P,Y}^{-1}(t))=\eta_{P,Z}^{-1}i_{\ast}g(t)=h.$$
Hence, $i^{\ast}P\in \mathcal{P}(\mathcal{A})$. It is proved dually that $i^{!}$ preserves injective objects. The proofs of $(3')$, $(4)$ and $(4')$ are similar.

(5) For any $X\in \mathcal{A}$, there exists a deflation $f:P\rightarrow i_{\ast}X$ with $P\in \mathcal{P}(\mathcal{B})$. Applying the functor $i^{\ast}$, we obtain a deflation $i^{\ast}P\rightarrow X$ with $i^{\ast}P\in \mathcal{P}(\mathcal{A})$. Hence $\mathcal{A}$ has enough projectives. Since $i^\ast$ preserves projectives, we have that $\add(i^{\ast}(\mathcal{P}(\mathcal{B})))\subseteq\mathcal{P}(\mathcal{A})$. Conversely, for $Q\in P(\mathcal{A})$, as above, there exists a deflation $i^{\ast}P\rightarrow Q$ with $P\in \mathcal{P}(\mathcal{B})$, which implies that $Q$ is a direct summand of $i^{\ast}P$. That is, $P(\mathcal{A})\subseteq \add (i^{\ast}(\mathcal{P}(\mathcal{B})))$. Similarly, the second statement in $(5)$ can be proved.

$(6)$ Since $j_\ast$ (resp. $j_{!}$) is exact, by (4) (resp. $(4')$), we obtain that $j^\ast$ preserves projectives (resp. injectives). Then, using the similar proof of $(5)$, we can prove $(6)$.


$(7)$ By $(5)$ we obtain that $\mathcal{A}$ has enough projectives. Since $i^!$ is  exact, we get that $i_\ast$ preserves projectives. Then according to Lemma \ref{Iso}(1), we prove $(7)$.

$(7')$ By using Lemma \ref{Iso}(1) immediately, we obtain the proof.

$(8)$ Let $X\stackrel{f}{\longrightarrow}Y\stackrel{g}{\longrightarrow}Z\stackrel{}\dashrightarrow$ be an  $\mathbb{E}_{\C}$-triangle. Since $j_{!}$ is right exact, there is an $\mathbb{E}_{\B}$-triangle $X'\stackrel{h_{2}}{\longrightarrow}j_{!}Y\stackrel{j_{!}g}{\longrightarrow}j_{!}Z\stackrel{}\dashrightarrow$ and  a compatible morphism $h_{1}$ such that $(h_{1},h_{2})\in\Phi_{ j_{!}f}$. Noting that $j^{\ast}j_{!}X\stackrel{j^{\ast}j_{!}f}{\longrightarrow}j^{\ast}j_{!}Y\stackrel{j^{\ast}j_{!}g}{\longrightarrow}j^{\ast}j_{!}Z\stackrel{}\dashrightarrow$ is an $\mathbb{E}_{\C}$-triangle since $j^{\ast}j_{!}\cong \Id_{\C}$, and $j^\ast j_!f=(j^\ast h_2)(j^\ast h_1$), we obtain that $j^*h_1$ is an inflation and then $j^{\ast}h_{1}$ is an isomorphism since $j^{\ast}h_{1}$ is a deflation and compatible. So $j^\ast X'\cong j^*j_!X$. Set $M=\cocone (h_{1})$, by Lemma \ref{2.5}(2), we have that $j^{\ast}M=0$. By (R2), there is an object $N\in \A$ such that $i_{\ast}N=M$.  Since $i^{\ast}$ is exact, $i^{\ast}X'\stackrel{i^{\ast}h_{2}}{\longrightarrow}i^{\ast}j_{!}Y\stackrel{i^{\ast}j_{!}g}{\longrightarrow}i^{\ast}j_{!}Z\stackrel{}\dashrightarrow$ is an $\mathbb{E}_{\A}$-triangle, which implies $i^{\ast}X'=0$ by $i^{\ast}j_{!}=0$. Similarly, since $i^{\ast}M\stackrel{}{\longrightarrow}i^*j_{!}X\stackrel{i^*h_1}{\longrightarrow}i^*X'\stackrel{}\dashrightarrow$ is an  $\mathbb{E}_{\A}$-triangle, we get that $i^{\ast}M=0$. Then $M=i_{\ast}N\cong i_{\ast}(i^{\ast}i_{\ast}N)\cong i_{\ast}i^{\ast}M=0$. Hence, $h_{1}$ is an isomorphism and $j_{!}X\stackrel{j_{!}f}{\longrightarrow}j_{!}Y\stackrel{j_{!}g}{\longrightarrow}j_{!}Z\stackrel{}\dashrightarrow$ is an  $\mathbb{E}_{\B}$-triangle. For the natural transformation, it is implied from the assumption that $j_!$ is right exact.
The proof of $(8')$ is similar.
\end{proof}

\begin{proposition}\label{triangulated}  Let ($\mathcal{A}$, $\mathcal{B}$, $\mathcal{C}$) be a recollement of extriangulated categories as (\ref{recolle}).

$(1)$ If $i^{!}$ is exact, for each $X\in\mathcal{B}$, there is an $\mathbb{E}_{\B}$-triangle
  \begin{equation*}\label{third}
  \xymatrix{i_\ast i^! X\ar[r]^-{\theta_X}&X\ar[r]^-{\vartheta_X}&j_\ast j^\ast X\ar@{-->}[r]&}
   \end{equation*}
 where $\theta_X$ and  $\vartheta_X$ are given by the adjunction morphisms.

$(2)$ If $i^{\ast}$ is exact, for each $X\in\mathcal{B}$, there is an $\mathbb{E}_{\B}$-triangle
  \begin{equation*}\label{four}
  \xymatrix{ j_! j^\ast X\ar[r]^-{\upsilon_X}&X\ar[r]^-{\nu_X}&i_\ast i^\ast X \ar@{-->}[r]&}
   \end{equation*}
where $\upsilon_X$ and $\nu_X$ are given by the adjunction morphisms.
\end{proposition}
\begin{proof} We only prove (1) since the proof of (2) is similar. By (R4), for each $X\in\mathcal{B}$, there is a left exact $\mathbb{E}_{\B}$-triangle sequence
  \begin{equation*}\label{five}
  \xymatrix{i_\ast i^! X\ar[r]^-{\theta_X}&X\ar[r]^-{\vartheta_X}&j_\ast j^\ast X\ar[r]^{h}&i_\ast A}
   \end{equation*}
such that $i_\ast i^! X\stackrel{\theta_X}{\longrightarrow}X\stackrel{h_{1}}{\longrightarrow}M\stackrel{}\dashrightarrow$ and $M\stackrel{h_{2}}{\longrightarrow}j_\ast j^\ast X\stackrel{h}{\longrightarrow}i_\ast A\stackrel{}\dashrightarrow$ are $\mathbb{E}_{\B}$-triangles, $h_{2}$ is compatible and $\vartheta_X=h_{2}h_{1}$. Since $i^{!}$ is exact, we obtain that $i^{!}i_\ast i^! X\stackrel{i^{!}\theta_X}{\longrightarrow}i^{!}X\stackrel{i^{!}\vartheta_X}{\longrightarrow}i^{!}j_\ast j^\ast X$ is left exact. By Lemma \ref{CY}(2), $i^{!}j_{\ast}j^{\ast}X=0$, so $i^{!}h_{2}$ is an isomorphism. Thus, $i^!M=0$. It follows that $A\cong i^! i_\ast A=0$. By Lemma \ref{2.5}(2), $h_{2}$ is an isomorphism and thus $i_\ast i^! X\stackrel{\theta_X}{\longrightarrow}X\stackrel{\vartheta_X}{\longrightarrow}j_\ast j^\ast X\stackrel{}\dashrightarrow$ is an $\mathbb{E}_{\B}$-triangle. This finishes the proof.
\end{proof}

In what follows, let us give an example of recollement of an extriangulated category which is neither abelian nor triangulated.
\begin{example}\label{fang}
Let $A$ be the path algebra of the quiver $1\stackrel{\alpha}\longrightarrow2$ over a field. The Auslander-Reiten quiver of $\mod A$ is as follows
\begin{equation*}
\xymatrix@!=1.0pc{    &  P_1\ar[dr]^{\psi}    &   \\
S_2\ar[ur]^{\varphi}  &      & S_1 .}
\end{equation*}
Then the triangular matrix algebra $B=\begin{pmatrix} A &A\\0  &A\end{pmatrix}$ is given by the quiver
$$\xymatrix{
   &   \cdot\ar[dr]^{\beta} &  &  \\
   \cdot\ar[ur]^{\alpha}\ar[dr]_{\delta} &    &  \cdot& \\
    &  \cdot \ar[ur]_{\gamma} &  & \\  }
$$
with the relation $\beta\alpha=\gamma\delta$. It is well-known that $\mod B$ can be identified with the morphism category of $A$-modules. That is, each $B$-module can be written as a triple ${X\choose Y}_f$ such that $f:Y\rightarrow X$ is a homomorphism of $A$-modules. In the following, we write ${X\choose Y}$ instead of ${X\choose Y}_0$. The Auslander-Reiten quiver of $B$ is given by
\begin{equation*}
\xymatrix@!=3.0pc{    & P_1\choose0\ar[dr] &   & 0\choose S_2 \ar[dr] &      & {S_1\choose S_1}_{1}\ar[dr]   &   \\
S_2\choose 0 \ar[dr]\ar[ur]&    & {P_1\choose S_2}_{\varphi}\ar[dr]\ar[ur]\ar[r]  & {P_1\choose P_1}_{1}\ar[r]  &  {S_1\choose P_1}_{\psi}\ar[ur]\ar[dr]    &    & 0\choose S_1  \\
 &  {S_2\choose S_2}_{1}\ar[ur]  &&  S_1\choose 0 \ar[ur]      &   &0\choose P_1\ar[ur] }
\end{equation*}
By \cite[Example 2.12]{Ps}, we have a recollement of abelian categories
\begin{equation}\label{exam}
  \xymatrix{\mod A\ar[rr]|{i_{*}}&&\ar@/_1pc/[ll]|{i^{*}}\ar@/^1pc/[ll]|{i^{!}}\mod B
\ar[rr]|{j^{\ast}}&&\ar@/_1pc/[ll]|{j_{!}}\ar@/^1pc/[ll]|{j_{\ast}} \mod A}
\end{equation}
such that $i^{*}({X\choose Y}_{f})=\rm Coker$$(f)$, $i_{*}(X)={X\choose 0}$, $i^{!}({X\choose Y}_{f})=X$, $j_{!}(Y)={Y\choose Y}_{1}$, $j^{\ast}({X\choose Y}_{f})=Y$ and  $j_{*}(Y)={0\choose Y}$.

Let $\mathcal{X}_{1}=\mod A$ and $\mathcal{X}_{2}=\add({S_1\oplus P_1})$.
Observe that $\mathcal{X}_{2}$ is an extriangulated category which is neither abelian nor triangulated. Indeed, $\mathcal{X}_{2}$ is an extension-closed subcategory of $\mod B$. On the one hand, $P_1$ is a non-zero injective object. This implies $\mathcal{X}_{2}$ is not triangulated. On the other hand, $\Ker \psi$ does not belong to $\mathcal{X}_{2}$. This implies $\mathcal{X}_{2}$ is not abelian.  In addition, let
$$\mathcal{X}=\add({S_2\choose0}\oplus{P_1\choose 0}\oplus{S_1\choose 0}\oplus  {P_1\choose P_1}_{1}\oplus {S_1\choose P_1}_{\psi}\oplus{S_1\choose S_1}_{1}\oplus {0\choose P_1}\oplus {0\choose S_1}),$$
and it is also an extriangulated category.

We claim that
\begin{equation}\label{lliz}
  \xymatrix{\mathcal{X}_{1}\ar[rr]|{i_{*}}&&\ar@/_1pc/[ll]|{i^{*}}\ar@/^1pc/[ll]|{i^{!}}\mathcal{X}
\ar[rr]|{j^{\ast}}&&\ar@/_1pc/[ll]|{j_{!}}\ar@/^1pc/[ll]|{j_{\ast}} \mathcal{X}_{2}}
\end{equation}
is a recollement of an extriangulated category which is neither abelian nor triangulated. In fact, one can check that $i_{*},j^{\ast}$ are exact functors, $i^{\ast}$, $j_!$ are right exact functors and $i^{!}$, $j_\ast$ are left exact functors.

(R1) For any ${X\choose X'}_{f}\in\mathcal{X}$ and $Y\in \mathcal{X}_{2}$, $$\Hom_{\mathcal{X}}(j_{!}Y, {X\choose X'}_{f})\cong \Hom_{\mod A}(Y,j^{\ast} {X\choose X'}_{f})= \Hom_{\mathcal{X}_{2}}(Y,j^{\ast} {X\choose X'}_{f})$$
and
$$\Hom_{\mathcal{X}_{2}}(j^{\ast} {X\choose X'}_{f},Y)\cong \Hom_{\mod B}( {X\choose X'}_{f},j_{\ast}Y)= \Hom_{\mathcal{X}}( {X\choose X'}_{f},j_{\ast}Y).$$
It follows that $(j_{!}, j^{\ast}, j_{\ast})$ is an adjoint triple. Similarly, so is $(i^{*}, i_{\ast}, i^{!})$.

(R2) Note that $\Im i_{\ast}=\Ker j^{\ast}=\add({P_1\choose 0}\oplus {S_2\choose0}\oplus {S_1\choose 0})$.

(R3) Since the functors $i_\ast$, $j_!$ and $j_\ast$ in (\ref{exam}) are fully faithful, it follows that $i_\ast$, $j_!$ and $j_\ast$ in (\ref{lliz}) are also fully faithful.

(R4) Note that $i^{!}$ is exact. For any ${X\choose Y}_{f}\in\mathcal{X}$, by \cite[Proposition 2.6]{Ps}, there exists an exact sequence
$$0{\longrightarrow}i_{\ast}i^{!}{ X\choose Y}_{f}\stackrel{}{\longrightarrow}{ X\choose Y}_{f}\stackrel{}{\longrightarrow} j_{\ast}j^{\ast}{ X\choose Y}_{f}{\longrightarrow}0$$
in $\mod B$, which also provides a left exact $\mathbb{E}$-triangle sequence
$${ X\choose 0}\stackrel{}{\longrightarrow}{ X\choose Y}_{f}\stackrel{}{\longrightarrow} { 0\choose Y}{\longrightarrow}0$$
in $\mathcal{X}$.

(R5) For ${X\choose Y}_{f}\in\mathcal{X}$, by \cite[Proposition 2.6]{Ps}, there exists an exact sequence
$$0{\longrightarrow}{ X'\choose 0}_{f}\stackrel{}{\longrightarrow}j_{!}j^{\ast}{ X\choose Y}_{f}\stackrel{}{\longrightarrow}{ X\choose Y}_{f}\stackrel{}{\longrightarrow}  i_{\ast}i^{\ast}{ X\choose Y}_{f}{\longrightarrow}0$$
in $\mod B$ with $X'\in \mod A$, which provides a right exact $\mathbb{E}$-triangle sequence
$${ X'\choose 0}\stackrel{}{\longrightarrow}\begin{pmatrix} Y\\Y\end{pmatrix}_{1}\stackrel{}{\longrightarrow}\begin{pmatrix} X\\Y\end{pmatrix}_{f}\stackrel{}{\longrightarrow}  { \Coker f\choose 0}$$
in $\mathcal{X}$.
\end{example}



\section{Glued cotorsion pairs}

First of all, let us recall the definition of cotorsion pairs in an extriangulated category.
\begin{definition}\cite[Definition 4.1]{Na}\label{Cotorsion} Let $\mathscr{C}$ be an extriangulated category and $\mathcal{T}$, $\mathcal{F}\subseteq\mathscr{C}$ be a pair of subcategories of $\mathscr{C}$.
The pair $(\mathcal{T},\mathcal{F})$ is called a {\em cotorsion pair} in $\mathscr{C}$ if it satisfies the following conditions:

$(a)$ $\mathbb{E}(\mathcal{T},\mathcal{F})=0.$

$(b)$ For any $C\in\mathscr{C}$, there exists a conflation $F\longrightarrow T\longrightarrow C$ such that $F\in\mathcal{F}$, $T\in\mathcal{T}$.

$(c)$ For any $C\in\mathscr{C}$, there exists a conflation $C\longrightarrow F'\longrightarrow T'$ such that $F'\in\mathcal{F}$, $T'\in\mathcal{T}$.

\end{definition}

\begin{remark}\label{remark}
Let $(\mathcal{U}$, $\mathcal{V})$ be a cotorsion pair in an extriangulated category $\mathscr{C}$. Then
\begin{itemize}
\item $M\in\mathcal{U}$ if and only if $\mathbb{E}(M,\mathcal{V})=0$;
\item $N\in\mathcal{V}$ if and only if $\mathbb{E}(\mathcal{U},N)=0$;
\item $\mathcal{U}$ and $\mathcal{V}$ are extension-closed;
\item $\mathcal{U}$ is contravariantly finite and $\mathcal{V}$ is covariantly finite in $\mathscr{C}$;
\item $\mathcal{P}\subseteq\mathcal{U}$ and $\mathcal{I}\subseteq\mathcal{V}$.
\end{itemize}
\end{remark}

\begin{definition}\label{glued}Let ($\mathcal{A}$, $\mathcal{B}$, $\mathcal{C}$) be a recollement of extriangulated categories as (\ref{recolle}). Given cotorsion pairs $(\mathcal{T}_{1},\mathcal{F}_{1})$ and $(\mathcal{T}_{2},\mathcal{F}_{2})$ in $\mathcal{A}$ and $\mathcal{C}$, respectively, set
\begin{equation*} \mathcal{T}=\{B\in \mathcal{B}~|~i^{\ast }B\in\mathcal{T}_{1}~\text{and}~j^{\ast}B\in \mathcal{T}_{2}  \}
\end{equation*}and
\begin{equation*} \mathcal{F}=\{B\in \mathcal{B}~|~i^{!}B\in\mathcal{F}_{1}~\text{and}~j^{\ast}B\in \mathcal{F}_{2}  \}.
\end{equation*}
In this case, we call $(\mathcal{T},\mathcal{F})$ the glued pair with respect to $(\mathcal{T}_{1},\mathcal{F}_{1})$ and $(\mathcal{T}_{2},\mathcal{F}_{2})$.
\end{definition}

For a subcategory $\mathcal{X}$ of $\mathscr{C}$, we define full subcategories
$$\mathcal{X}^{\perp_{1}}=\{M\in \mathscr{C}|~\mathbb{E}(\mathcal{X},M)=0 \}.$$
Recall that an extriangulated category $\mathscr{C}$ is called {\em Frobenius} if $\mathscr{C}$ has enough projectives and enough injectives and moreover the projectives coincide with the injectives. Note that each triangulated category is a Frobenius extriangulated category.

Now, we can give our main result of this paper as the following
\begin{theorem}\label{main}Let $(\mathcal{A},\mathcal{B},\mathcal{C})$ be a recollement of extriangulated categories as (\ref{recolle}), and $(\mathcal{T}_{1},\mathcal{F}_{1})$ and $(\mathcal{T}_{2},\mathcal{F}_{2})$ be two cotorsion pairs in $\mathcal{A}$ and $\mathcal{C}$, respectively. Let $(\mathcal{T},\mathcal{F})$ be the glued pair with respect to $(\mathcal{T}_{1},\mathcal{F}_{1})$ and $(\mathcal{T}_{2},\mathcal{F}_{2})$. Assume that $\mathcal{B}$ has enough projectives and $i^{!}$, $j_{!}$ are exact.

$(1)$ If one of the following conditions holds:
\begin{itemize}
  \item [(i)] $\mathbb{E}_\mathcal{B}(\mathcal{T},\mathcal{F})=0$.
  \item [(ii)] $i^{\ast}$ is  exact.
  \item [(iii)] For any morphism $f:i_{\ast}A\rightarrow j_{!}T$ with $A\in \mathcal{A}$ and $T\in \mathcal{T}_{2}$, the induced map $f^\ast:~\B(j_{!}T,F)\rightarrow \B(i_{\ast}A,F)$ is surjective for any $F\in \mathcal{F}$.
  \item [(iv)]$\mathcal{T}\subseteq j_{!}\mathcal{T}_{2}$ or $i_{\ast}\mathcal{F}_{1}\subseteq \mathcal{T}^{\perp_{1}}$.
    \item [(v)] $\mathcal{A}$ and $\mathcal{B}$ are Frobenius extriangulated categories.
\end{itemize}
Then $(\mathcal{T},\mathcal{F})$ is a cotorsion pair in $\mathcal{B}$. 

$(2)$ If $(\mathcal{U},\mathcal{V})$ is a cotorsion pair in $\mathcal{B}$ such that $i_{\ast}i^{!}\mathcal{U}\subseteq \mathcal{U}$ and $i_{\ast}i^{\ast}\mathcal{U}\subseteq\mathcal{U}$, then $(i^{\ast}\mathcal{U},i^{!}\mathcal{V})$ is a cotorsion pair in $\mathcal{A}$.

$(3)$ If $(\mathcal{U},\mathcal{V})$ is a cotorsion pair in $\mathcal{B}$ such that $j_{\ast}j^{\ast}\mathcal{V}\subseteq\mathcal{V}$ or $j_{!}j^{\ast}\mathcal{U}\subseteq\mathcal{U}$, then $(j^{\ast}\mathcal{U},j^{\ast}\mathcal{V})$ is a cotorsion pair in $\mathcal{C}$.
\end{theorem}

In what follows, we say that a commutative diagram is {\em exact} if every sub-diagram of the form $X\rightarrow Y\rightarrow Z$ is a conflation. Before proving Theorem \ref{main}, we give the following

\begin{lemma}\label{T} Keep the notation as Definition \ref{glued}.

$(1)$ If $i^{!}$ is exact, then $(\mathcal{T},\mathcal{F})$  satisfies $(b)$ in Definition \ref{Cotorsion}.

$(2)$ If $i^{!}$ and $j_{!}$ are exact, then $(\mathcal{T},\mathcal{F})$  satisfies $(c)$ in Definition \ref{Cotorsion}.
\end{lemma}
\begin{proof}
(1) For any $M\in \mathcal{B}$, there exists an $\mathbb{E}_\mathcal{C}$-triangle $F_{2}\stackrel{}{\longrightarrow}\stackrel{}T_{2}{\longrightarrow}j^{\ast}M\stackrel{}\dashrightarrow$ with $F_{2}\in \mathcal{F}_{2}$ and $T_{2}\in \mathcal{T}_{2}$, since $(\mathcal{T}_{2},\mathcal{F}_{2})$ is a cotorsion pair in $\mathcal{C}$. Since $i^{!}$ is exact, by Lemma \ref{CY}$(8')$, $j_{\ast}$ is exact. Applying $j_{\ast}$ to the above $\mathbb{E}_\mathcal{C}$-triangle, we obtain an $\mathbb{E}_\mathcal{B}$-triangle $j_{\ast}F_{2}\stackrel{}{\longrightarrow}j_{\ast}T_{2}\stackrel{}{\longrightarrow}j_{\ast}j^{\ast}M\stackrel{}\dashrightarrow$. Consider the following commutative diagram
\begin{equation}\label{11}
\xymatrix{
  j_{\ast}F_{2}\ar@{=}[d] \ar[r] & H\ar[d]\ar[r] & M \ar[d]_{\eta_{M}} \ar@{-->}[r]^-{\eta_M^*\delta}  &  \\
  j_{\ast}F_{2} \ar[r]^-{} & j_{\ast}T_{2}\ar[r]^-{} & j_{\ast}j^{\ast}M \ar@{-->}[r]^-\delta&  }
\end{equation}
where $\eta_{M}$ is the unit of the adjoint pair $(j^{\ast},j_{\ast})$. Applying $j^{\ast}$ to (\ref{11}) and using Lemma \ref{adjoint}(1), we obtain that $j^{\ast}H\cong T_{2}$.  We also have an $\mathbb{E}_\mathcal{A}$-triangle $F_{1}\stackrel{}{\longrightarrow}T_{1}\stackrel{}{\longrightarrow}i^{\ast}H\stackrel{}\dashrightarrow$ with $F_{1}\in \mathcal{F}_{1}$ and $T_{1}\in \mathcal{T}_{1}$ since $(\mathcal{T}_{1},\mathcal{F}_{1})$ is a cotorsion pair in $\mathcal{A}$. Then $i_{\ast}F_{1}\stackrel{}{\longrightarrow}i_{\ast}T_{1}\stackrel{}{\longrightarrow}i_{\ast}i^{\ast}H\stackrel{}\dashrightarrow$ is an $\mathbb{E}_\mathcal{B}$-triangle, since $i_\ast$ is exact. For $H\in \mathcal{B}$, by (R5), there exists a commutative diagram
\begin{equation*}\label{SS}
\xymatrix{
  &i_{\ast}A' \ar[r]&j_{!}j^{\ast}H\ar[rr]^-{\upsilon_H}\ar[dr]_{h_{2}}&  &H\ar[r]^-{h}&i_{\ast}i^{\ast}H &\\
           &                &       &  K \ar[ur]_{h_{1}}& }
\end{equation*}
in $\mathcal{B}$ such that $i_{\ast}A'\stackrel{}{\longrightarrow}j_{!}j^{\ast}H\stackrel{h_{2}}{\longrightarrow}K\stackrel{}\dashrightarrow$ and $K\stackrel{h_{1}}{\longrightarrow}H\stackrel{h}{\longrightarrow}i_{\ast}i^{\ast}H\stackrel{}\dashrightarrow$ are $\mathbb{E}_\mathcal{B}$-triangles and $h_{2}$ is compatible, moreover, $j_{!}j^{\ast}H\stackrel{v_H}{\longrightarrow}H\stackrel{h}{\longrightarrow}i_{\ast}i^{\ast}H$ is right exact.
By \cite[Proposition 3.15]{Na}, we have the following exact commutative diagram
\begin{equation}\label{4.10}
\xymatrix{   &   & K \ar[d]_{t}\ar@{=}[r] &  K\ar[d]^{h_{1}} &\\
  & i_{\ast}F_{1} \ar[r]\ar@{=}[d]  &  T \ar[d]_{l} \ar[r]&  H \ar[d]^-h \\                                               & i_{\ast}F_{1} \ar[r]  &  i_{\ast}T_{1}  \ar[r]&  i_{\ast}i^{\ast}H . & }
\end{equation}
Consider the following commutative diagram
\begin{equation}\label{K}
\xymatrix{
  j_{!}j^{\ast}H\ar[rr]\ar[dr]_-{h_{2}}&  &T\ar[r]^-{l}&i_{\ast}T_{1} &\\
           &           K \ar[ur]_{t}      &       & & }
\end{equation}
where $h_{2}$ is a deflation and compatible, and $t$ is an inflation. Thus, the first row of (\ref{K}) is right exact. Applying the right exact functor $i^{\ast}$ to (\ref{K}) and using Lemma \ref{zero}(2), we obtain that $i^{\ast}T\cong i^{\ast}i_{\ast} T_{1}\cong T_{1}$. Applying $j^{\ast}$ to the $\mathbb{E}_\mathcal{B}$-triangle in the second row of (\ref{4.10}), we have that  $j^{\ast}T\cong j^{\ast}H\cong T_{2}$. Hence, $T\in\mathcal{T}$. Applying $\rm (ET4)^{op}$ yields an exact commutative diagram
\begin{equation}\label{7}
\xymatrix{
  i_{\ast}F_{1} \ar@{=}[d] \ar[r] & F  \ar[d] \ar[r] & j_{\ast}F_{2}  \ar[d] \\
  i_{\ast}F_{1} \ar[r] & T \ar[d] \ar[r] &  H \ar[d] \\
  &  M  \ar@{=}[r] & M. }
\end{equation}
Applying $j^*$ to the $\mathbb{E}_\mathcal{B}$-triangle in the first row of (\ref{7}), we obtain that $j^{\ast}F\cong j^{\ast}j_{\ast}F_{2}\cong F_{2}$. Similarly, applying $i^!$ to the $\mathbb{E}_\mathcal{B}$-triangle in the first row of (\ref{7}), by the dual of Lemma \ref{zero}, we have that $i^{!}F\cong i^{!}i_{\ast}F_{1}\cong F_{1}$. That is, $F\in\mathcal{F}$. So the second column in (\ref{7}) gives a desired $\mathbb{E}_\mathcal{B}$-triangle. Therefore, $(\mathcal{T},\mathcal{F})$ satisfies $(b)$ in Definition \ref{Cotorsion}.

(2) For any $M\in \mathcal{B}$, there exists an $\mathbb{E}_\mathcal{C}$-triangle $j^{\ast}M\stackrel{}{\longrightarrow}F_{2}\stackrel{}{\longrightarrow}\stackrel{}T_{2}\dashrightarrow$ with $F_{2}\in \mathcal{F}_{2}$ and $T_{2}\in \mathcal{T}_{2}$, since $(\mathcal{T}_{2},\mathcal{F}_{2})$ is a cotorsion pair in $\mathcal{C}$. Since $j_{!}$ is exact, we obtain an $\mathbb{E}_\mathcal{B}$-triangle $j_{!}j^{\ast}M\stackrel{}{\longrightarrow}j_{!}F_{2}\stackrel{}{\longrightarrow} j_{!}T_{2}\dashrightarrow$.
Consider the following exact commutative diagram
\begin{equation}\label{20}
\xymatrix{
 j_{!}j^{\ast}M \ar[d]_{\epsilon_{M}} \ar[r] & j_{!}F_{2}\ar[d]\ar[r] & j_{!}T_{2} \ar@{=}[d]\ar@{-->}[r]^-{\delta}  &  \\
 M \ar[r]&  H  \ar[r]  &      j_{!}T_{2} \ar@{-->}[r]^-{(\epsilon_{M})_{\ast}\delta}&  }
\end{equation}
where $\epsilon_{M}$ is the counit of the adjoint pair $(j_{!},j^{\ast})$.  Applying $j^{\ast}$ to (\ref{20}) and using Lemma \ref{adjoint}(2), we obtain that $j^{\ast}H\cong F_{2}$. We also have an $\mathbb{E}_\mathcal{A}$-triangle $i^{!}H\stackrel{}{\longrightarrow}F_{1}\stackrel{}{\longrightarrow}T_{1}\stackrel{}\dashrightarrow$ with $F_{1}\in \mathcal{F}_{1}$ and $T_{1}\in \mathcal{T}_{1}$, since $(\mathcal{T}_{1},\mathcal{F}_{1})$ is a cotorsion pair in $\mathcal{A}$. Then $i_{\ast}i^{!}H\stackrel{}{\longrightarrow}i_{\ast}F_{1}\stackrel{}{\longrightarrow}i_{\ast}T_{1}\stackrel{}\dashrightarrow$ is an  $\mathbb{E}_\mathcal{B}$-triangle since $i_{\ast}$ is exact.
Consider the following exact commutative diagram
\begin{equation}\label{30}
\xymatrix{
 i_{\ast}i^{!}H \ar[d]_{\epsilon_{H}} \ar[r] &i_{\ast}F_{1}\ar[d]\ar[r] & i_{\ast}T_{1} \ar@{=}[d]\ar@{-->}[r]^-{\delta}  &  \\
 H \ar[r]&  F  \ar[r]  &       i_{\ast}T_{1} \ar@{-->}[r]^-{(\epsilon_{H})_{\ast}\delta}&  }
\end{equation}
where $\epsilon_{H}$ is the counit of the adjoint pair $(i_{\ast},i^{!})$. Applying $i^{!}$ to (\ref{30}) and using Lemma \ref{adjoint}(2), we obtain that $i^{!}F\cong F_{1}$.  Applying $j^{\ast}$ to the $\mathbb{E}_\mathcal{B}$-triangle in the second row of (\ref{30}) and using $j^{\ast}i_{\ast}=0$, we obtain that $j^{\ast}F\cong j^{\ast}H\cong F_{2}$. Thus, $F\in\mathcal{ F}$.
Applying $\rm (ET4)$ yields an exact commutative diagram
\begin{equation}\label{60}
\xymatrix{
  M \ar@{=}[d] \ar[r] & H  \ar[d] \ar[r] & j_{!}T_{2}  \ar[d] \\
 M \ar[r] & F \ar[d] \ar[r] &  T \ar[d] \\
  &  i_{\ast}T_{1}  \ar@{=}[r] & i_{\ast}T_{1}. }
\end{equation}
Applying $j^{\ast}$ to the $\mathbb{E}_\mathcal{B}$-triangle in the third column of (\ref{60}), we obtain that $j^{\ast}T\cong j^{\ast}j_{!}T_{2}\cong T_{2} $. Similarly, applying $i^{\ast}$ to the $\mathbb{E}_\mathcal{B}$-triangle in the third column of (\ref{60}), by Lemma \ref{zero}(2), we have that $i^{\ast}T\cong i^{\ast}i_{\ast}T_{1}\cong T_{1}$. That is, $T\in\mathcal{T}$. So the second row in (\ref{60}) gives a desired $\mathbb{E}_\mathcal{B}$-triangle. Therefore, $(\mathcal{T},\mathcal{F})$ satisfies $(c)$ in Definition \ref{Cotorsion}.
\end{proof}


Now we are in the position to prove Theorem \ref{main}.

\textbf{{Proof of Theorem \ref{main}.}}
By Lemma \ref{CY}$(8')$, $(5)$ and $(6)$, we have that $j_{\ast}$ is exact and $\mathcal{A}$, $\mathcal{C}$ has enough projectives.

(1) Take any objects $T\in\mathcal{T}$ and $F\in\mathcal{F}$. By Lemma \ref{T}, we only need to prove $\mathbb{E}_\mathcal{B}(T,F)=0$.

(i) Immediately.

(ii) Since $i^{\ast}$ is exact, by Proposition \ref{triangulated}(2), there is an $\mathbb{E}_\mathcal{B}$-triangle $j_! j^\ast T\stackrel{}{\longrightarrow}T\stackrel{}{\longrightarrow}i_\ast i^\ast T\stackrel{}\dashrightarrow$. Applying $\Hom_\mathcal{B}(-,F)$ to this $\mathbb{E}_\mathcal{B}$-triangle, we get an exact sequence
$$\mathbb{E}_\mathcal{B}(i_{\ast}i^{\ast}T,F){\longrightarrow}\mathbb{E}_\mathcal{B}(T,F){\longrightarrow}\mathbb{E}_\mathcal{B}(j_{!}j^{\ast}T,F).$$
By Lemma \ref{CY}(7), we obtain that $\mathbb{E}_\mathcal{B}(i_{\ast}i^{\ast}T,F)\cong\mathbb{E}_\mathcal{A}(i^{\ast}T,i^{!}F)=0$, since $i^{\ast}T\in\mathcal{T}_1$ and $i^{!}F\in\mathcal{F}_1$. Similarly, $\mathbb{E}_\mathcal{B}(j_{!}j^{\ast}T,F)\cong\mathbb{E}_\mathcal{C}(j^{\ast}T,j^{\ast}F)=0$. It follows that $\mathbb{E}_\mathcal{B}(T,F)=0$. 

(iii)
By (R5), there exists a commutative diagram
\begin{equation*}\label{SSS}
\xymatrix{
  &i_{\ast}A' \ar[r]&j_{!}j^{\ast}T\ar[rr]^-{}\ar[dr]&  &T\ar[r]&i_{\ast}i^{\ast}T\ar[r] &\\
           &                &       &  K \ar[ur]& }
\end{equation*}
in $\mathcal{B}$ such that $i_{\ast}A'\stackrel{}{\longrightarrow}j_{!}j^{\ast}T\stackrel{}{\longrightarrow}K\stackrel{}\dashrightarrow$ and $K\stackrel{}{\longrightarrow}T\stackrel{}{\longrightarrow}i_{\ast}i^{\ast}T\stackrel{}\dashrightarrow$ are $\mathbb{E}_\mathcal{B}$-triangles. Applying $\Hom(-,F)$ to the two $\mathbb{E}_\mathcal{B}$-triangles above, we have two exact sequences
$$\mathbb{E}_\mathcal{B}(i_{\ast}i^{\ast}T,F)\rightarrow\mathbb{E}_\mathcal{B}(T,F)\rightarrow\mathbb{E}_\mathcal{B}(K,F)$$
and
$$\Hom_{\B}(j_{!}j^{\ast}T,F)\stackrel{f^{\ast}}\rightarrow \Hom_{\B}(i_{\ast}A' ,F)\rightarrow\mathbb{E}_\mathcal{B}(K,F)\rightarrow\mathbb{E}_\mathcal{B}(j_{!}j^{\ast}T,F).$$
By Lemma \ref{CY}$(7')$, we obtain that $\mathbb{E}_\mathcal{B}(j_{!}j^{\ast}T,F)\cong\mathbb{E}_\mathcal{C}(j^{\ast}T,j^{\ast}F)=0$. Similarly, $\mathbb{E}_\mathcal{B}(i_{\ast}i^{\ast}T,F)\cong\mathbb{E}_\mathcal{A}(i^{\ast}T,i^{!}F)=0$. By hypothesis, $f^{\ast}$ is surjective, so we obtain that $\mathbb{E}_{\B}(K,F)=0$. It follows that $\mathbb{E}_\mathcal{B}(T,F)=0$.

(iv)
Since $i^{!}$ is exact, by Proposition \ref{triangulated}(1), there is an $\mathbb{E}_\mathcal{B}$-triangle \begin{equation*}\label{i}
i_\ast i^! F\stackrel{}{\longrightarrow}F\stackrel{}{\longrightarrow}j_\ast j^\ast F\stackrel{}\dashrightarrow.
\end{equation*}
Applying $\Hom(T,-)$ to the $\mathbb{E}_\mathcal{B}$-triangle above, we have an exact sequence
$$\mathbb{E}_\mathcal{B}(T,i_\ast i^! F){\longrightarrow}\mathbb{E}_\mathcal{B}(T,F){\longrightarrow}\mathbb{E}_\mathcal{B}(T,j_\ast j^\ast F).$$
Since $i^!$ is exact, we obtain that $j_\ast$ is exact and then $j^\ast$ preserves projectives by Lemma \ref{CY}(4). Thus, by Lemma \ref{Iso}(1), we have that $\mathbb{E}_\mathcal{B}(T,j_\ast j^\ast F)\cong\mathbb{E}_\mathcal{C}(j^\ast T,j^\ast F)=0$. If $\mathcal{T}\subseteq j_{!}\mathcal{T}_{2}$, i.e., there exists $T_2\in\mathcal{T}_{2}$ such that $T\cong j_{!}T_2$. Then $\mathbb{E}_\mathcal{B}(T,i_\ast i^! F)\cong\mathbb{E}_\mathcal{B}(j_{!}{T}_{2},i_\ast i^! F)\cong\mathbb{E}_\mathcal{C}({T}_{2},j^{\ast}i_\ast i^! F)=0$. If $i_{\ast}\mathcal{F}_{1}\subseteq \mathcal{T}^{\perp_{1}}$, then $\mathbb{E}_\mathcal{B}(T,i_{\ast}i^{!}F)=0$ since $i^!F\in\mathcal{F}_{1}$. Hence, $\mathbb{E}_\mathcal{B}(T,F)=0$.

(v) Since $\mathcal{A}$ and $\mathcal{B}$ are Frobenius extriangulated categories, projectives coincide with injectives, by Lemma \ref{CY}(4), $i_{\ast}$ preserves injectives. Then by Lemma \ref{Iso}(2), we have the isomorphism $\mathbb{E}_{\mathcal{A}}(i^\ast X,Y)\cong\mathbb{E}_{\mathcal{B}}(X,i_\ast Y)$ for any $X\in\mathcal{B}$ and $Y\in\mathcal{A}$. Using the similar proof of (iv), we prove (v).

(2) For any $X\in \mathcal{A}$, there exists an $\mathbb{E}_\mathcal{B}$-triangle $V\stackrel{}{\longrightarrow}U\stackrel{}{\longrightarrow}i_{\ast}X\stackrel{}\dashrightarrow$ with $V\in \mathcal{V}$ and $U\in \mathcal{U}$. Then we have an $\mathbb{E}_\mathcal{A}$-triangle $i^{!}V\stackrel{}{\longrightarrow}i^{!}U\stackrel{}{\longrightarrow}X\stackrel{}\dashrightarrow$, since $i^{!}$ is exact. Since $i_{\ast}i^{!}U\in i_{\ast}i^{!}\mathcal{U}\subseteq \mathcal{U}$, we obtain that $i^{!}U\in i^{\ast}\mathcal{U}$. That is, $(i^{\ast}\mathcal{U},i^{!}\mathcal{V})$ satisfies $(b)$ in Definition \ref{Cotorsion}. Dually, we can prove that  $(i^{\ast}\mathcal{U},i^{!}\mathcal{V})$ also satisfies $(c)$ in Definition \ref{Cotorsion}. For any $U\in\mathcal{U}$ and $V\in\mathcal{V}$, by Lemma \ref{CY}(7), we have that $\mathbb{E}_\mathcal{A}(i^{\ast}{U},i^{!}{V})\cong\mathbb{E}_\mathcal{B}(i_{\ast}i^{\ast}{U},{V})\cong\mathbb{E}_\mathcal{B}(U',{V})=0$ for some $U'\in\mathcal{U}$.
Hence, $(i^{\ast}\mathcal{U},i^{!}\mathcal{V})$ is a cotorsion pair in $\mathcal{A}$.

(3) The proof is analogous to (2).\fin\\

Applying the main theorem to recollements of triangulated categories, we have the following
\begin{corollary}\cite[Theorem 3.1]{Chen} Let ($\mathcal{A}$, $\mathcal{B}$, $\mathcal{C}$) be a recollement of triangulated categories. Let $(\mathcal{T},\mathcal{F})$ be a  glued pair with respect to cotorsion pairs $(\mathcal{T}_{1},\mathcal{F}_{1})$ and $(\mathcal{T}_{2},\mathcal{F}_{2})$ in $\mathcal{A}$ and $\mathcal{C}$, respectively. Then $(\mathcal{T},\mathcal{F})$ is a cotorsion pair in $\mathcal{B}$.
\end{corollary}

 We finish this section with a straightforward example illustrating Theorem \ref{main}.
\begin{example}\label{example} Keep the notation as Example \ref{fang}.

Observe that $\mathcal{P}(\mod A)=\add({P_{1}\oplus S_{2}})$ and $\mathcal{I}(\mod A)=\add({P_{1}\oplus S_{1}})$. By Remark \ref{remark}, we know that $\mod A$ has only two cotorsion pairs
$$\mathcal{H}_{1}=(\mathcal{P}(\mod A),\mod A)~~\text{and}~~\mathcal{H}_{2}=(\mod A,\mathcal{I}(\mod A)).$$

(1) Let $(\mathcal{T},\mathcal{F})$ be the glued pair with respect to $\mathcal{H}_{1}$ and $\mathcal{H}_{1}$. Then
$$\mathcal{T}=\add({S_{2}\choose0}\oplus {P_{1}\choose0}\oplus {S_{2}\choose S_{2}}_{1}\oplus {0\choose S_{2}}\oplus {P_{1}\choose P_{1}}_{1}\oplus {S_{1}\choose P_{1}}_{\psi}\oplus {0\choose P_{1}})$$
and $\mathcal{F}=\mod \mathcal{B}$.

(2) Let $(\mathcal{T},\mathcal{F})$ be the glued pair with respect to $\mathcal{H}_{1}$ and $\mathcal{H}_{2}$. Then
$$\mathcal{T}=\mod \mathcal{B}\setminus\add({P_{1}\choose S_{2}}_{\varphi}\oplus {S_{1}\choose0})$$
and
$$\mathcal{F}=\mod \mathcal{B}\setminus\add({S_{2}\choose S_{2}}_{1}\oplus {P_{1}\choose S_{2}}_{\varphi}\oplus {0\choose S_{2}}).$$

(3) Let $(\mathcal{T},\mathcal{F})$ be the glued pair with respect to $\mathcal{H}_{2}$ and $\mathcal{H}_{2}$. Then $\mathcal{T}=\mod \mathcal{B}$ and
$$\mathcal{F}=\mod \mathcal{B}\setminus\add({S_{2}\choose S_{2}}_{1}\oplus {P_{1}\choose S_{2}}_{\varphi}\oplus {0\choose S_{2}}).$$

(4) Let $(\mathcal{T},\mathcal{F})$ be the glued pair with respect to $\mathcal{H}_{2}$ and $\mathcal{H}_{1}$.
 Then
$$\mathcal{T}=\mod \mathcal{B}\setminus\add({S_{1}\choose S_{1}}_{1}\oplus {0\choose S_{1}})$$
and
$$\mathcal{F}=\mod \mathcal{B}\setminus\add({S_{2}\choose 0}\oplus {S_{2}\choose S_{2}}_{1}).$$

However, every glued pair above is not a cotorsion pair in $\mod B$ since $\rm Ext^{1}(\mathcal{T},\mathcal{F})\neq0$.

Observe that
$$\mathcal{P}(\mod B)=\add({S_{2}\choose0}\oplus {P_{1}\choose0}\oplus {S_{2}\choose S_{2}}_{1} \oplus {P_{1}\choose P_{1}}_{1} )$$
and
$$\mathcal{I}(\mod B)=\add({P_{1}\choose P_{1}}_{1}\oplus {S_{1}\choose S_{1}}_{1}\oplus {0\choose P_{1}} \oplus {0\choose S_{1}}) .$$

(5) Take $\mathcal{T}=\mathcal{P}(\mod B)$ and $\mathcal{F}=\mod B$, then $(\mathcal{T},\mathcal{F})$ is a cotorsion pair in $\mod B$. One can check that
$$i_{\ast}i^{\ast}\mathcal{T}=i_{\ast}i^{!}\mathcal{T}=\add({S_{2}\choose 0}\oplus {P_{1}\choose 0} )\subseteq\mathcal{T}$$
and $j_{\ast}j^{\ast}\mathcal{F}\subseteq \mathcal{F}.$ Therefore, by Theorem \ref{main}(2) and (3), we know that $(i^{\ast}\mathcal{T},i^{!}\mathcal{F})$ and $(j^{\ast}\mathcal{T},j^{\ast}\mathcal{F})$ are cotorsion pairs in $\mod A$.

$(5')$ Let $$\mathcal{T}=\add({S_{2}\choose0}\oplus {P_{1}\choose0} \oplus {S_{2}\choose S_{2}}_{1}\oplus {P_{1}\choose P_{1}}_{1}\oplus {0\choose S_{1}})$$
and
$$\mathcal{F}=\add({P_{1}\choose P_{1}}_{1}\oplus {S_{1}\choose S_{1}}_{1} \oplus {0\choose P_{1}}\oplus {0\choose S_{1}}\oplus {P_{1}\choose 0}\oplus {S_{2}\choose S_{2}}_{1}\oplus {S_{2}\choose 0}),$$
then $(\mathcal{T},\mathcal{F})$ is also a cotorsion pair in $\mod B$.
We have that
$$i_{\ast}i^{\ast}\mathcal{T}=i_{\ast}i^{!}\mathcal{T}=\add({S_{2}\choose 0}\oplus {P_{1}\choose 0} )\subseteq\mathcal{T}.$$
Thus, by Theorem \ref{main}(2), we obtain that $(i^{\ast}\mathcal{T},i^{!}\mathcal{F})$ is a cotorsion pairs in $\mod A$. But $(j^{\ast}\mathcal{T},j^{\ast}\mathcal{F})$ is not a cotorsion pair in $\mod A$ since $ \Ext^{1}(j^{\ast}\mathcal{T},j^{\ast}\mathcal{F})\neq0$.

\end{example}

\end{document}